\documentclass[12pt]{article} 
\usepackage{makeidx}
\usepackage[a4paper]{geometry} 
\usepackage{amsthm, amstext, amssymb, amsmath}
\usepackage{cite}

\usepackage{epsfig}
\usepackage{listings} 
\usepackage{algorithm2e}

\newtheorem{theorem}{Theorem}[section]
\newtheorem{lemma}[theorem]{Lemma}
\newtheorem{corollary}[theorem]{Corollary}
\newtheorem{definition}[theorem]{Definition}

\newtheorem{proposition}[theorem]{Proposition} 

\def\GG{\mathcal{G}} 
\def\cC{\mathcal{C}} 
\def\bF{\mathbf{F}}
\def\G{\mathbf{G}}
\def\I{\mathcal{I}}
\def\bI{\mathbf{I}}
\def\M{\mathbf{M}}
\def\cM{\mathcal{M}}
\def\F{\mathbb{F}}
\def\N{\mathbb{N}}
\def\Q{\mathbb{Q}}
\def\R{\mathbb{R}}

\def\phi{\varphi}

\def\cl{\mbox{cl}} 
\def\to{\rightarrow}
\def\td{\mbox{td}}
\def\FOl{\mathrm{FO}(\lambda)}
\def\FOkl{\mathrm{FO}_k(\lambda)}
\def\vk{v_1, \ldots, v_k}
\def\se{\subseteq}

\def\bd{\mathrm{bd}}
\def\decomp{depth-decomposition} 

\newcommand\term[1]{{\em #1}}

\begin{document}
\title{First order convergence of matroids\thanks{
Research supported by the European Research Council under the European Union's Seventh Framework Programme (FP7/2007- 2013)/ERC grant agreement no.~259385.
The work of the second author was also supported by the Engineering and Physical Sciences Research Council Standard Grant number EP/M025365/1.}}
\author{
Franti\v{s}ek Kardo\v{s}\thanks{LaBRI, University of Bordeaux, France. E-mail: {\tt frantisek.kardos@labri.fr}.}\and
Daniel Kr\'al'\thanks{Mathematics Institute, DIMAP and Department of Computer Science, University of Warwick, Coventry CV4 7AL, UK. E-mail: {\tt d.kral@warwick.ac.uk}.}\and
Anita Liebenau\thanks{School of Mathematical Sciences, Monash University, 9 Rainforest Walk, Clayton 3800, Australia. E-mail: {\tt Anita.Liebenau@monash.edu}. This work was done while this author was affiliated with Department of Computer Science and DIMAP, University of Warwick, Coventry CV4 7AL, UK.}\and 
Luk\'a\v s Mach\thanks{Department of Computer Science and DIMAP, University of Warwick, Coventry CV4 7AL, UK. E-mail: {\tt lukas.mach@gmail.com}.}}
\date{}
\maketitle 

\begin{abstract} 
The model theory based notion of the first order convergence
unifies the notions of the left-convergence for dense structures and
the Benjamini-Schramm convergence for sparse structures.
It is known that every first order convergent sequence of
graphs with bounded tree-depth can be represented by an analytic limit object called a limit modeling.
We establish the matroid counterpart of this result:
every first order convergent sequence of matroids with bounded branch-depth representable over a fixed finite
field has a limit modeling, i.e., there exists an infinite matroid with the elements
forming a probability space that has asymptotically the same first order properties.
We show that neither of the bounded branch-depth assumption nor the representability assumption can be removed.
\end{abstract} 


\section{Introduction}

The theory of combinatorial limits keeps attracting a growing amount of attention.
Combinatorial limits have sparked many exciting developments in extremal combinatorics, in theoretical computer science, and other areas.
Their significance is also evidenced by a recent monograph of Lov\'asz~\cite{bib-lovasz-book}.
The better understood case of convergence of dense structures originated in the series of papers
by Borgs, Chayes, Lov\'asz, S\'os, Szegedy, and Vesztergombi~\cite{bib-borgs08+,bib-borgs+,bib-borgs06+,bib-lovasz06+,bib-lovasz10+}
on the dense graph convergence, and the notion was applied in various settings including
hypergraphs, partial orders, permutations, and tournaments~\cite{bib-elek12+,bib-hladky12+,bib-hoppen13+,bib-hoppen11+,bib-janson11,bib-kral13+}.
The convergence of sparse structures (such as graphs with bounded maximum degree) referred
to as the Benjamini-Schramm convergence~\cite{bib-aldous07+,bib-benjamini01+,bib-elek07,bib-hatami+}
is less understood despite having links to many problems of high importance.
For example, the conjecture of Aldous and Lyons~\cite{bib-aldous07+} on Benjamini-Schramm convergent sequences of graphs
is essentially equivalent to Gromov's question of whether all countable discrete groups are sofic. 
Other notions of convergence of sparse graphs were also proposed and studied 
\cite{br2011, bcg2013,bckl2013,bib-elek07,bib-hatami+}.

In the light of many results on the convergence of graphs,
one can ask whether a reasonable theory of matroid convergence can be developed. 
The first obstacle to building such a theory comes from the fact that matroids
when viewed as hypergraphs (e.g.~with edges being the bases)
can be too sparse for the classical dense convergence approach to be directly applied, and
too dense for the sparse convergence approach at the same time.
For example, the number of bases of the graphic matroid of $K_n$ is $n^{n-2}$,
an exponentially small fraction of all $(n-1)$-element subsets of the edge set of $K_n$ and
an even tinier fraction of all subsets of the edge set,
which rules out the dense convergence approach.
On the other hand, each element of this matroid is contained in a non-constant number of bases, and
it is impossible to follow the sparse convergence approach.
We overcome this obstacle by adapting the notion of the first order convergence to matroids.

The notion of the first order convergence was introduced by Ne\v set\v ril and Ossona de Mendez~\cite{bib-folim1,bib-folim2}
as an attempt to unify the convergence notions in the dense and sparse settings:
a sequence of structures of a fixed type (e.g.,~graphs) is {\em first order convergent}
if the probability that a random $k$-tuple of its elements has a first order property $\varphi$,
converges for every choice of $\varphi$ (a formal definition can be found in Subsection~\ref{subsect-FO}).
It holds that every first order convergent sequence of dense structures is convergent in the dense sense and
every first order convergent sequence of sparse structures is convergent in the Benjamini-Schramm sense.

In the analogy to graphons in the setting of dense graphs and graphings in the setting of sparse graphs,
an analytic limit object called a {\em limit modeling} was proposed in~\cite{bib-folim1,bib-folim2}
to represent asymptotic properties of first order convergent sequences.
Unlike in the dense and sparse graph settings,
it is not true that every first order convergent sequence of graphs has a limit modeling.
For example, the sequence of Erd\H os-R\'enyi random graphs $G_{n,p}$ for $p\in (0,1)$
is first order convergent with probability one but it has no limit modeling \cite[Lemma 18]{bib-folim2}. 
In the same paper, Ne\v set\v ril and Ossona de Mendez showed the following.
\begin{theorem}
\label{thm-treedepth}
Every first order convergent sequence of graphs with bounded tree-depth has a limit modeling.
\end{theorem}
This result was extended to first convergent sequences of trees and
graphs of bounded path-width~\cite{bib-folim-ours,bib-folim3}.
Ne\v set\v ril and Ossona de Mendez~\cite{bib-folim4} have recently shown that
every first order convergent sequence of graphs from a nowhere-dense class of graphs has a limit modeling,
which is the most general result possible for monotone classes of graphs~\cite[Theorem 25]{bib-folim2}.

As a test that the approach to the matroid convergence based on the first order convergence is meaningful,
it seems natural to prove the analogue of Theorem~\ref{thm-treedepth},
which is actually one of our results (Theorem~\ref{thm-branchdepth}).
On the way towards Theorem~\ref{thm-branchdepth},
we need to find a matroid parameter that can play the role of the graph tree-depth.
We do so by introducing a parameter called branch-depth in Section~\ref{sect-branch-depth}. 
We believe that this matroid parameter is the right analogue of the graph tree-depth because it has the following properties,
which we establish in this paper.
We refer the reader to \cite[Chapter 6]{bib-nom2012} for a thorough discussion of the graph tree-depth.
\begin{itemize}
\item The branch-depth of a matroid corresponding to a graph $G$ is at most the tree-depth of $G$.
\item The branch-depth of a matroid corresponding to a graph $G$ with tree-depth $d$ is at least $\frac{1}{2}\log_2d$ if $G$ is $2$-connected.
\item The branch-depth is a minor monotone parameter (the same holds for graph tree-depth).
\item The branch-depth of a matroid is at most the square of the length of its longest circuit (recall that the tree-depth of a graph $G$ is at most the length of its longest path).
\item The branch-depth of a matroid is at least the binary logarithm of the length of its longest circuit (recall that the tree-depth of a graph $G$ is at least the binary logarithm of the length of its longest path).
\end{itemize}
In addition, there exists an efficient algorithm that given an integer $d$ and
an oracle-represented input matroid either outputs its decomposition of bounded depth or
it determines that the branch-depth of the input matroid exceeds $d$.

Equipped with the notion of branch-depth, we prove the following theorem in Section~\ref{sect-modeling}.
\begin{theorem}
\label{thm-branchdepth}
Every first order convergent sequence of matroids with bounded branch-depth that is representable over a fixed finite field has a limit matroid modeling.  
\end{theorem}
\noindent Note that matroids representable over finite fields have a structure more similar to graphs than general matroids and
it is not surprising that Theorem~\ref{thm-branchdepth} includes this assumption.
In fact, we show in Section~\ref{sect-non-exist} that neither the assumption on the bounded branch-depth
nor the assumption on the representability over a fixed finite field can be dropped. In particular,
we construct a first order convergent sequence of binary matroids that has no limit modeling, and
a first order convergent sequence of rank three matroids representable over rationals that has no limit modeling.

\section{Notation} 
In this section, we introduce the notation used throughout the paper.

\subsection{Finite matroids}\label{subs:finMatroids}
We start by introducing concepts related to finite matroids.
We refer to the monograph by Oxley \cite{o2011} for a more detailed treatment. 
In Subsection \ref{subs:InfMatroids}, we extend the terminology to infinite matroids.  
A {\em matroid} $M$ is a pair $(E, \I)$ where $E$ is a finite set, called the {\em ground set}, 
and $\I \subseteq 2^E$ is a collection of its subsets referred to as {\em independent sets}. 
The set $\I$ is required to be nonempty, to be hereditary (i.e., for every $F\in \I$, $\I$ must contain 
every subset of $F$), and to satisfy {\em the augmentation axiom}: 
if $F$ and $F'$ are independent sets with $|F| < |F'|$, 
then there exists $x \in F'\setminus F$ such that $F \cup \{x\} \in \I$.
We abuse the notation and we often denote by $M$ the ground set of the matroid $M$. 

A subset $F\se M$ is called {\em dependent} if $F\not\in \I$, 
and a minimal dependent set is a {\em circuit}. 
The number of elements of a circuit is referred to as its {\em length}.
It is well-known that the collection $\cC$ of circuits of a matroid $M$ satisfies the following properties. 
\begin{itemize}
\item[(C1)] $\emptyset \not\in \cC$.
\item[(C2)] If $C_1,C_2 \in \cC$ and $C_1\se C_2$, then $C_1=C_2$. 
\item[(C3)] If $C_1,C_2 \in \cC$ with $e\in C_1\cap C_2$ and $f\in C_1\setminus C_2$,  
	then there exists a circuit $C_3\in \cC$ such that $f\in C_3 \se (C_1\cup C_2)\setminus \{e\}$. 
\end{itemize}
Furthermore, (C1)--(C3) form an alternative set of axioms to define matroids. More precisely, 
a collection $\cC$ of subsets of $M$ is the collection of circuits of a matroid if and only if 
it satisfies (C1)--(C3).
The {\em rank} $r_M(F)$ of a set $F \subseteq M$ is the size of the largest independent subset of $F$. 
The {\em rank of a matroid} $r_M(M)$ is the rank of the ground set of $M$.
It is well-known that the rank function of a matroid $M$ is submodular, i.e.,~for any 
two subsets $F_1,F_2\se M$ it holds that $r_M(F_1\cup F_2)+r_M(F_1\cap F_2)\leq r_M(F_1)+r_M(F_2)$. 
When there is no danger of confusion, we omit the subscript, i.e., we just use $r(F)$ instead of $r_M(F)$.

There are two particular important examples of matroids. 
A {\em graphic matroid} $M(G)$ is obtained from a graph $G$ in the following way:  
the elements of $M(G)$ are the edges of $G$, and a set of edges is independent if it is acyclic. 
{\em Vector matroids} have a set of vectors of a vector space as their ground set, 
and a set of elements is independent if  they are linearly independent. 

A matroid $M$ is called {\em representable over a field $\F$} if  
there exists a function $f$ that maps the elements of $M$  to vectors over $\F$ such that 
$F\se M$ is independent in $M$ if and only if $f(F)$ is linearly independent.  
A matroid is {\em binary} iff it is representable over the binary field $\F_2$. 

A {\em loop} is an element $e$ of $M$ with $r(\{e\}) = 0$, 
and a {\em bridge} is an element such that $r(M \backslash e) = r(M) - 1$.
Two elements $e$ and $e'$ are {\em parallel} if neither of them is a loop and $r(\{e,e'\})=1$.
If $F$ is a subset of elements of $M$, the {\em closure of $F$} is defined as 
$\cl(F) := \big\{x : r(F \cup \{ x \}) = r(F) \big\}.$ 
Clearly, $r(\cl(F))=r(F)$.

If $F$ is a subset of the elements of $M$, then $M\backslash F$ is the matroid 
obtained from $M$ by {\em deleting} the elements of $F$, i.e., the elements of 
$M\backslash F$ are those not contained in $F$, and a set $F'$ is independent 
in $M\backslash F$ iff it is independent in $M$. 
The matroid $M/F$ is obtained by {\em contracting} $F$: the elements of $M/F$ 
are those not contained in $F$, and a subset $F'$ of such elements is independent in 
$M/F$ iff $F'$ is independent in $M$ and $r(F\cup F')=r(F)+r(F')$. 
When $F=\{e\}$ is a single element, we write $M\backslash e$ and $M/e$ instead of 
$M\backslash \{e\}$ and $M/\{e\}$. 
The restriction $M|F$ of a matroid $M$ to $F$ is the matroid 
$M\setminus \overline{F}$, where $\overline{F}$ denotes the complement of $F$ in $M$.
Finally, a {\em minor} of a matroid $M$ is
a matroid obtained by a sequence of deleting and contracting some of its elements.
It is not hard to show that if a graph $G'$ is a minor of a graph $G$,
then the matroid $M(G')$ is a minor of the matroid $M(G)$.

A matroid $M$ is {\em connected} if the only two subsets $F\se M$ satisfying $r(F)+r(\overline{F}) = r(M)$ 
are the empty set and the whole ground set.
A {\em component of $M$} is a set $F$ that is an inclusion-wise maximal subset such that $M|F$ is connected.
The components of $M$ are equivalence classes given by the binary relation that
represents that two elements of $M$ are contained in a common circuit.
Hence, any two components of a matroid $M$ are disjoint.
If $M$ is a matroid and $M_1,\ldots,M_k$ are its components,
then $r_M(X)=r_{M_1}(X\cap M_1)+\cdots+r_{M_k}(X\cap M_k)$.

\subsection{Matroid algorithms}
Algorithms for matroids have been studied extensively, and we want to review selected important facts here. 
It is common (see, e.g.,~\cite{c.gaven, rw1980, c.seymour}) to assume that the input matroid 
is presented by means of an independence oracle. That is, we assume that we can determine whether any 
subset of the elements of the given matroid is independent using a black-box function in unit time. 
The complexity of algorithms for matroids is measured in terms of the number of elements of the input matroid. 

There is an efficient algorithm \cite{t1960} to test whether a given binary matroid is graphic, 
and if so to find a suitable graph. However, in general, deciding if an oracle-given matroid is 
binary cannot be solved in subexponential time \cite{c.seymour}. In the same paper \cite{c.seymour}, 
Seymour presents a polynomial-time algorithm to decide whether an oracle-given matroid 
is graphic. 

A lot of decision problems for matroids involve the structural matroid parameter {\em branch-width}. 
Rather than giving the exact definition here, let us just say that matroid branch-width is the analogue of 
graph tree-width.
Oum and Seymour \cite{os2007} showed, improving \cite{h2005}, that 
for every fixed $k\geq 1$, it can be decided in polynomial time whether the branch-width of an 
oracle-given matroid is at most $k$, and
that an optimal branch-decomposition can be constructed (for such matroids). 

If $\cM$ is a class of matroids that are representable over a fixed finite field and
that have branch-width bounded by a constant,
then properties expressible in monadic second order over $\cM$ can be decided in cubic time 
\cite{c.hlineny}. In \cite{c.gaven}, Gaven\v{c}iak, Oum, and the second author 
 introduce the notion of locally bounded branch-width and present a fixed parameter algorithm 
 to decide first order properties on the class of regular matroids with locally bounded 
branch-width. 

Hlin\v{e}n\'y \cite{h2006-2} also showed that for
every field $\F$ of order at least four, it is NP-hard to decide whether a matroid given by its 
rational representation
is $\F$-representable. The result still holds when restricting the input to matroids of branch-width at most three.  
On the other hand, for every $k\geq 1$ and any two finite fields $\F$ and $\F'$, 
there is a polynomial-time algorithm that decides whether a given $\F$-representable matroid of branch-width 
at most $k$ is also $\F'$-representable \cite{k2007}.  

Similarly to the relation between the tree-depth and tree-width,
the branch-width of a matroid $M$ is upper bounded by the branch-depth of $M$ up to an additive constant.
It is natural to ask whether Theorem~\ref{thm-branchdepth} can be extended to sequences of matroids
representable over finite fields that have bounded branch-width;
we believe that this is likely to be the case but it might be challenging to prove
since the analogous statement for sequences of graphs has been proven only very recently~\cite{bib-folim4}.

\subsection{Infinite matroids}
\label{subs:InfMatroids}
One of the ways to define the notion of infinite matroids is to require the augmentation 
axiom to hold for finite subsets and to additionally require that an infinite set $F\se E$ is 
independent if and only if all of its finite subsets are independent.
Such matroids are called {\em finitary}.  
The drawback of this definition is that finitary matroids can have only finite circuits. 

A robust notion of infinite matroids was proposed by Bruhn et~al.~\cite{bdkpw2013}.
They developed five equivalent axiom systems that characterize (infinite) matroids through independent sets, bases, 
the closure operator, circuits, and their rank function.
We present the characterization through circuits here. 
Let $M$ be a set, and let $\cC\se 2^{M}$ be a collection of subsets. 
Further, let  $\I=\I(\cC)$ denote the {\em $\cC$-independent sets}, that is the collection of sets $I\se M$ 
such that $C\nsubseteq I$ for all $C\in \cC$. 
A family $\cC$ is the collection of circuits of a matroid
if it satisfies (C1), (C2), and the following two conditions.  
\begin{itemize}
\item[(C3')] Whenever $X\se C \in \cC$ and $\{C_e\, | \, e\in X\}$ is a family of elements of $\cC$ such that 
	$e\in C_{f} \Leftrightarrow e=f$ for all $e,f \in X$, then for every $f \in C\setminus \left(\bigcup_{e\in X}  C_e\right)$ 
	there exists $C'\in \cC$ such that $f\in C' \se \left(C\cup \bigcup_{e\in X}  C_e\right)\setminus X$. 
\item[(CM)] Whenever $I\se F \se E$ and $I\in\I$, the set $\{I'\in\I\ :\ I\se I' \se F \}$ has a maximal element. 
\end{itemize} 
If $|X|=1$, the axiom (C3') becomes the usual {\em strong circuit elimination axiom} (C3).  
We remark that all finitary matroids are matroids in the sense just defined.

\subsection{First order convergence}
\label{subsect-FO}

For a set $\lambda$ of relational symbols, let $\FOl$ denote the set of first order formulas 
using symbols from $\lambda$, and let $\FOkl\se\FOl$ denote the set of all such formulas $\phi$ with $k$ free variables. 
A {\em $\lambda$-modeling} (or just a {\em modeling} if $\lambda$ is clear from the context) $A$
is a (finite or infinite) $\lambda$-structure 
whose domain is a standard Borel space equipped with a probability measure $\nu$ such that the following holds:
for every $\varphi\in\FOkl$,
the subset $A^\varphi\subseteq A^k$ formed by all $k$-tuples of the elements of $A$ satisfying $\varphi$
is measurable with respect to the product measure $\nu^k$.

For a formula $\phi \in \FOkl$ and a modeling $A$, 
the {\em Stone pairing} $\langle\phi,A\rangle$ is $\nu^k(A^\varphi)$,
i.e., the Stone pairing is the probability that a randomly chosen $k$-tuple of the elements of $A$ satisfies $\phi$. 
When a finite $\lambda$-structure $A$ with $|A|$ elements is viewed as a modeling with a uniform discrete probability measure,
it holds that
$$\langle\phi, A\rangle = \frac{A^\varphi}{|A|^k} = \frac{\big|\{(\vk)\in A^k: A \models \phi[\vk]\}\big|}{|A|^k}.$$ 
A sequence $(A_n)_{n\in \N}$ of finite $\lambda$-structures is {\em first order convergent} if the sequence 
$\langle\phi,A_n\rangle$ converges for every first order formula $\phi\in\FOl$. 
A $\lambda$-modeling $A$ is a {\em limit modeling} of a 
first order convergent sequence $(A_n)_{n\in\N}$ if 
$$ \langle\phi,A \rangle = \lim_{n \to \infty} \langle\phi,A_n\rangle$$ 
for every formula $\phi\in\FOl$. 

In Sections~\ref{sect-branch-depth} and~\ref{sect-modeling},
we work with rooted trees, rooted forests and their mode\-lings,
which were studied in \cite[Part 3]{bib-folim2}.
A {\em rooted tree} is a tree with a distinguished vertex referred to as the root, and
a {\em rooted forest} is a graph such that each of its component is a rooted tree.
The {\em depth} of a rooted tree is the length of the longest path from the root to a leaf, and
the depth of a rooted forest is the maximum depth of a rooted tree contained in it.
Rooted forests can be described by a language with a single binary relation representing the parent-child relation.
In addition, we also consider rooted forests with vertices colored with one of a bounded number of colors.
The vertex coloring of a rooted forest that uses $k$ colors
can be described by extending the language with $k$ unary relations, each representing one of the colors.
We use the following~\cite[Theorem 34]{bib-folim2} to prove one of our main results. 
\begin{theorem}\label{limForests} 
Let $k$ and $d$ be fixed integers.
Every first order convergent sequence $(F_n)_{n\in\N}$ of rooted forests with depth at most $d$ and
with vertices colored with at most $k$ colors has a limit modeling. 
\end{theorem}
Theorem~\ref{limForests} is proven in~\cite{bib-folim2}
for rooted forests described by a language
with a single symmetric binary relation representing edges and
a single unary relation distinguishing roots of trees in the forest.
Since there is a basic interpretation scheme translating the description of rooted forests
from this language to the language that we consider here and
there is also a basic interpretation scheme in the other direction (see the next subsection for a definition if needed),
Theorem~34 from~\cite{bib-folim2} and Theorem~\ref{limForests} are equivalent by \cite[Propositions 3 and 4]{bib-folim2}.

Our paper concerns matroids and their modelings;
we now introduce the notation related to the first order convergence matroids and matroid modelings.
Let $\lambda_M$ be the countable language containing a $k$-ary relation $I_k$ for every positive integer $k$.
The relation $I_k$ is formed by all $k$-tuples of elements that are independent in the matroid.
Finite matroids can be axiomatized in the first order language by a countable set of $\lambda_M$-formulas.
However, in general, the axioms (C1), (C2), (C3') and (CM) cannot be replaced by a countable 
set of first order axioms. 

Let $(M_n)_{n\in \N}$ be a sequence of finite matroids, equipped with the uniform measure on its element sets. 
We define $(M_n)_{n\in \N}$ to be first order convergent
if the sequence of the Stone pairings $\langle\phi,M_n\rangle$ converges for every first order $\lambda_M$-formula $\phi$.
A $\lambda_M$-modeling $\M$ is a {\em limit modeling} of $(M_n)_{n\in \N}$
if it is an infinite matroid and 
$$ \langle\phi,\M \rangle = \lim_{n \to \infty} \langle\phi,M_n\rangle$$ 
for every first order $\lambda_M$-formula $\phi$.
Note that this definition is stronger than that of a limit modeling because
we require additionally that the limit modeling is an infinite matroid. 
Note that if there exists an integer $K$ such that every circuit of $M_n$ has length at most $K$,
then every circuit of the limit modeling $\M$ has length at most $K$.
In particular, $\M$ is finitary. 

\subsection{Interpretation schemes}

Let $\kappa, \lambda$ be signatures, where $\lambda$ has $q$ relational symbols 
$R_1,\ldots,R_q$ with respective arities $r_1,\ldots,r_q$. An {\em interpretation scheme} $\bI$ 
of $\lambda$-structures in $\kappa$-structures is defined by 
an integer $k$, which is called the {\em exponent} of the interpretation scheme, 
a formula $\theta_E\in \mathrm{FO}_{2k}(\kappa)$, 
a formula $\theta_0\in \mathrm{FO}_k(\kappa)$, 
and a formula $\theta_i\in \mathrm{FO}_{r_ik}(\kappa)$ for each symbol $R_i\in \lambda$, such that: 
\begin{itemize} 
\item the formula $\theta_E$ defines an equivalence relation on $k$-tuples;
\item each formula $\theta_i$ is compatible with $\theta_E$, in the sense that for every $0\leq i \leq q$ it holds 
	$$ \bigwedge_{1\leq j\leq r_i} \theta_E({\bf x}_j,{\bf y}_j)\quad \vdash \quad \theta_i({\bf x}_1,\ldots,{\bf x}_{r_i}) \leftrightarrow \theta_i({\bf y}_1,\ldots,{\bf y}_{r_i}), $$ 
	where $r_0=1$, ${\bf x}_j$ and ${\bf y}_j$ represent $k$-tuples of free variables, and
	$\theta_i({\bf x}_1,\ldots,{\bf x}_{r_i})$ stands for $\theta_i(x_{1,1},\ldots,x_{1,k},\ldots, x_{r_i,1},\ldots, x_{r_i,k})$.
\end{itemize}
For a $\kappa$-structure $A$, we denote by $\bI(A)$ the $\lambda$-structure $B$ defined as follows:
\begin{itemize}
\item the domain of $B$ is the subset of the $\theta_E$-equivalence classes $[{\bf x}]\se A^k$ of the tuples ${\bf x}=(x_1,\ldots,x_k)$ 
	such that $A\models \theta_0({\bf x})$;
\item for each $1\leq i \leq q$ and every ${\bf v}_1,\ldots,{\bf v}_{r_i}\in A^{r_i\times k}$ such that $A\models \theta_0({\bf v}_j)$ 
	for every $1\leq j\leq r_i$ it holds 
	$$ B \models R_i ([{\bf v}_1],\ldots,[{\bf v}_{r_i}]) \quad \Longleftrightarrow \quad A\models \theta_i({\bf v}_1,\ldots,{\bf v}_{r_i}).$$
\end{itemize}
If $\theta_0$ is a tautology and $\theta_E$ is the equality on $k$-tuples,
then the interpretation scheme is said to be {\em basic}.

The following is a standard result.
\begin{proposition}\label{aux089}
Let $\bI$ be an interpretation scheme of $\lambda$-structures in $\kappa$-structures. 
Then there is a mapping $\tilde{\bI}:\mathrm{FO}(\lambda) \to \mathrm{FO}(\kappa)$ 
(defined by means of the formulas $\theta_E,\theta_0,\ldots,\theta_q$ above) such that for every $\phi\in\mathrm{FO}_p(\lambda)$, 
and every $\kappa$-structure $A$, the following property holds. 
For every $[{\bf v}_1],\ldots,[{\bf v}_{p}] \in \bI(A)^p$ (where ${\bf v}_i=(v_{i,1},\ldots,v_{i,k})\in A^k$) it holds 
$$ \bI(A) \models \phi ([{\bf v}_1],\ldots,[{\bf v}_{p}]) \quad \Longleftrightarrow \quad A\models \tilde{\bI}(\phi)({\bf v}_1,\ldots,{\bf v}_{p}).$$
\end{proposition}

We need the following generalization of Propositions 3 and 4 from~\cite{bib-folim2}.
\begin{lemma}\label{aux671}
Let $\bI$ be an interpretation scheme of $\lambda$-structures in $\kappa$-structures, 
let $\tilde{\bI}$ be the mapping from Proposition \ref{aux089}, 
and let $(A_n)_{n\in \N}$ be a sequence of finite $\kappa$-structures such that 
\begin{itemize}
\item[(1)] $\lim\limits_{n\to\infty} \langle \theta_0,A_n\rangle = 1$, and
\item[(2)] $\lim\limits_{n\to\infty} \langle \phi, \bI(A_n)\rangle = \lim\limits_{n\to\infty} \langle \tilde{\bI}(\phi), A_n\rangle$
           for every $\phi\in\FOl$.
\end{itemize}
If the sequence $(A_n)_{n\in \N}$ is first order convergent,
then the sequence $(\bI(A_n))_{n\in \N}$ is also first order convergent.
Moreover, if ${\bf A}$ is a limit modeling of $(A_n)_{n\in \N}$,
then $\bI({\bf A})$ is a limit modeling of $(\bI(A_n))_{n\in \N}$.
\end{lemma}

\begin{proof} 
The sequence $(\bI(A_n))_{n\in \N}$ is first order convergent by \cite[Proposition 3]{bib-folim2}. 
Let $\Sigma_{\bf A}$ be the underlying $\sigma$-algebra of the probability space of ${\bf A}$.  
We define the $\sigma$-algebra on ${\bf B}:=\bI({\bf A})$ as follows:
$$ X\in \Sigma_{\bf B} \quad \text{ if and only if } \quad \bigcup_{x\in X} [x]\in\Sigma_{\bf A}.$$
The probability measure $\nu_{\bf B}$ on ${\bf B}$ is defined as
$\nu_{\bf B}(X):= \nu_{\bf A}(\bigcup_{x\in X} [x])$ for $X\in \Sigma_{\bf B}$, where $\nu_{\bf A}$ is 
the probability measure on ${\bf A}$. 
The condition (1) implies that $\nu_{\bf B}$ is indeed a probability measure, i.e.,~$\nu_{\bf B}({\bf B})=1$.

We now argue that $\bI({\bf A})$ is a limit modeling of the sequence $(\bI(A_n))_{n\in \N}$.  
Let $\phi \in \FOkl$. 
Proposition~\ref{aux089} yields that $(x_1,\ldots,x_k)\in\bI({\bf A})^\varphi$ iff
$([x_1],\ldots,[x_k])\in {\bf A}^{\tilde{\bI}(\phi)}$.
It follows from the definition of the $\sigma$-algebra $\Sigma_{\bf B}$ that $\bI({\bf A})^\varphi$ is measurable.
The definition of $\nu_{\bf B}$ yields that that 
$$ \langle \phi, \bI({\bf A})\rangle=\langle \tilde{\bI}(\phi), {\bf A}\rangle,$$ 
which combines with the condition (2) to the following:
$$ \langle \phi, \bI({\bf A})\rangle=\langle \tilde{\bI}(\phi), {\bf A}\rangle =
  \lim_{n\to\infty} \langle \tilde{\bI}(\phi), A_n\rangle=\lim_{n\to\infty} \langle \phi, \bI(A_n)\rangle.$$
Hence, $\bI({\bf A})$ is a limit modeling of the sequence $(\bI(A_n))_{n\in \N}$.
\qedhere 
\end{proof}

\section{Matroid branch-depth} 
\label{sect-branch-depth} 

In this section, we introduce a matroid parameter analogous to the graph tree-depth.  
We also present an algorithm that efficiently computes an approximate value of the parameter
of an input matroid together with the certifying depth-decomposition. 

\subsection{Definition and basic properties} 

The branch-depth of a matroid is equal to the optimal height of a certain kind of a decomposition tree.
In the definition below and in the proofs of subsequent claims, we use $\|T\|$ to denote the number of edges of a tree $T$.

\begin{definition}
\label{d.decomp} 
Let $M$ be a finite matroid. 
A \term{depth-decomposition of $M$} is a pair $(T,f)$, where $T$ is a rooted tree and $f: M \to V(T)$ is a mapping such that 
\begin{itemize}
\item[(1)] $r(M) = \|T\|$, and 
\item[(2)] $r(X) \le \|T^*(X)\|$ for every $X\subseteq M$, 
\end{itemize}
where $T^*(X)$ is the union of paths from the root to all the vertices
in $f(X)$.
The \term{branch-depth of a matroid $M$}, denoted by $\bd(M)$, is the smallest depth of its depth-decomposition,
i.e., the smallest depth of a rooted tree $T$ such that $(T,f)$ is a depth-decomposition of $M$.
\end{definition} 
For any matroid $M$ there is a trivial decomposition where the tree is a path of length $r(M)$
with one of its end vertices being the root and all the elements of $M$ mapped to the other end vertex.
The following lemma gives us a way to modify a depth-decomposition.
\begin{lemma}
\label{l.refined}
Let  $M$ be a finite matroid. If $(T,f)$ is a depth-decomposition of $M$, then there is a depth-decomposition $(T,f')$ 
	such that $f'(e)$ is a leaf of $T$ for every element $e$ of $M$. 
\end{lemma}
\begin{proof}
Let $(T,f)$ be a depth-decomposition of $M$. 
For every inner vertex $v$ of $T$, let $\ell(v)$ be a leaf of $T$ that is a descendant of $v$. 
For every $e\in M$, define $f'$ as follows. 
$$f'(e):=\begin{cases}
f(e)& \text{if } f(e) \text{ is a leaf of $T$, and}\\
\ell(f(e)) & \text{otherwise.}
\end{cases}$$
We now verify that $(T,f')$ is a depth-decomposition of $M$.
The part (1) of Definition~\ref{d.decomp} holds since we have not changed the tree $T$.
To check part (2), observe that for any subset $X$ of the elements of $M$,
the subtree $T^*(X)$ with respect to $f$ is contained in the subtree $T^*(X)$ with respect to $f'$.
Hence, $(T,f')$ is a depth-decomposition of $M$.
\end{proof}

As the notion of graph tree-depth \cite{bib-nom2012},
the parameter of matroid branch-depth is also minor monotone. 
\begin{proposition} 
\label{p.minor}
If $M'$ is a minor of $M$, then $\bd(M') \le \bd(M)$.
\end{proposition}
\begin{proof} 
Since a minor of a matroid is obtained by a sequence of contractions and deletions of some of its elements,
it is enough to show that if $M$ is a matroid and $e$ is an element of $M$, then the branch-depth of both $M / e$ and $M \setminus e$ is at most $\bd(M)$. 
Fix a matroid $M$ and $e \in M$. 
Let $(T,f)$ be a \decomp\ of $M$ of depth $\bd(M)$. 
By Lemma \ref{l.refined}, we can assume that $f(e)$ is a leaf of $T$ for every $e\in M$. 

If $e$ is a loop in $M$ then $M_0 := M \setminus e = M / e$. It is easy to see that for every $X\subseteq M_0$ we have $r_{M_0}(X)=r_{M}(X)$. 
Hence, $(T,f|_{M_0})$ is a \decomp\ of $M_0$. 

We now assume that $e$ is not a loop. 
Let $M_1 := M / e$, 
let $u$ be the leaf $f(e)$, and
let $v$ be the parent of $u$. 
Set $T_1 = T \setminus u$ and define $f_1: M_1 \to V(T_1)$ as follows:
$$
f_1(x)=\left\{
\begin{array}{ll}
v    & \textrm{if $f(x)=u,$ and} \\
f(x) & \textrm{otherwise.}
\end{array}\right.
$$
We now show that $(T_1,f_1)$ is a \decomp\ of $M_1$. 
Since $e$ is not a loop, we have $r(M_1)=r(M)-1$. 
Thus, $\|T_1\| = r(M_1)$.
Now, consider a subset $X\subseteq M_1$. 
Recall that $r_{M_1}(X) = r_M(X\cup\{e\}) - 1$.
If $u \in f(X),$ we employ the bound on the rank function provided by the \decomp\ of $M$: 
$$
\|T_1^*(X)\| = \|T^*(X\cup\{e\})\|-1 \ge r_M(X\cup\{e\})-1 = r_{M_1}(X).$$
Otherwise, we have
$$\|T_1^*(X)\| = \|T^*(X)\|\ge r_M(X)\ge r_{M_1}(X).$$

Let $M_2 = M \setminus e$. 
If $e$ is a bridge then $M \setminus e = M/e$. 
Hence, we may assume that $e$ is not a bridge in $M$. 
In this case, we claim that $(T,f|_{M_2})$ is a \decomp\ of $M_2$. 
Since the rank of $M_2$ equals the rank of $M$,
we have $r(M_2) = \|T\|$, and
it also holds that $\|T^*(X)\|\ge r_M(X) = r_{M_2}(X)$ for every $X\subseteq M_2$. 
\end{proof}

If a graph $G$ has a path with $n$ vertices, its tree-depth (see Definition~\ref{d.treedepth} if needed)
is at least $\lceil\log_2 n+1\rceil-1$, see e.g.~\cite[Chapter 6]{bib-nom2012}.
The next proposition relates the length of circuits in a matroid to its branch-depth,
in the analogy to the relation between the graph tree-depth and the existence of long paths. 

\begin{proposition}
\label{p.circuit}
Let $M$ be a matroid and $d$ the size of its largest circuit. 
Then $\bd(M) \ge \log_2 d $.
\end{proposition}

\begin{proof} 
Let $C_d$ be the matroid  that consists of exactly one circuit of size $d$.
If $M$ has a circuit of length $d$, then $M$ contains $C_d$ as a minor.
Hence, it is enough to show by Proposition~\ref{p.minor} that the branch-depth of $C_d$ is at least $\log_2 d$.
We prove this statement by induction on $d$.

Let $(T,f)$ be a \decomp\ of $C_d$ such that $T$ has depth $\bd(C_d)$ and such that $f(e)$ is a leaf of $T$ for every $e\in C_d$.
Its existence follows from Lemma~\ref{l.refined}. 

Let $w$ be the root of $T$. We first prove that the degree of $w$ is one.
Suppose not. Let $W$ be vertices of one of the subtrees of $w$, and
let $T_1$ be the subtree induced by $W\cup\{w\}$ and $T_2$ the subtree induced by the vertices not contained in $W$.
Observe that $\|T_i\|\ge r(f^{-1}(V(T_i))) = |f^{-1}(V(T_i))|$ for $i\in \{1,2\}$. 
It follows that $r(C_d)= \|T\| = \|T_1\| + \|T_2\| \ge |f^{-1}(V(T_1))|+|f^{-1}(V(T_2))| = |C_d| = r(C_d) + 1$,
which is impossible.

Let $v$ be a vertex of $T$ of degree larger than two that is as close to the root $w$ as possible. 
If there is no such vertex, $T$ is a path and it has depth $d - 1 \ge \log_2(d)$.

Let $P$ be a path from $w$ to $v$, $\ell$ its length, and $W$ vertices of one of the subtrees of $v$.
Let $T_1$ be the subtree induced by $W\cup\{v\}$ and $T_2$ the subtree induced by $v$ and the vertices not contained in $W$ or in $P$.
Further, let $m_i=\|T_i\|$ and $n_i=|f^{-1}(V(T_i))|$ for $i\in \{1,2\}$. Observe that both $n_1$ and $n_2$ are non-zero, and that 
\begin{equation}
m_1+m_2+\ell = r(C_d) = d-1 \qquad \textrm{and}\qquad n_1+n_2 = |C_d| = d.
\label{e.thmcircuit}
\end{equation} 
Since $f^{-1}(V(T_i))$ is a proper subset of $C_d$,
it is independent and we get that $n_i\le m_i+\ell$ for $i\in\{1,2\}$.
This yields that $m_i\le n_i-1$; otherwise, $n_{3-i}>m_{3-i}+\ell$ by (\ref{e.thmcircuit}).
By symmetry, we may assume that $n_1\le n_2$, which gives $n_1\le \frac{d}2$.

Let $M':=C_d/f^{-1}(V(T_1))$. Observe that $M'$ is isomorphic to $C_{n_2}$.
Let $T'$ be the tree obtained by considering a path of length of $m_1+\ell-n_1\ge 0$,
identifying one of its end vertices with the root of the tree $T_2$ and
rooting the resulting tree at the other end vertex of the path.
Observe that $\|T'\|=r(M')$ since the number of its edges is smaller by $r(V(T_1))$ compared to $M$, and
that the tree $T'$ with $f|_{M'}$ is a \decomp\ of $M'$. 
By induction, the depth of $T'$ is at least $\log_2 \frac{d}{2}$.
Since $m_1+\ell-n_1=\ell-(n_1-m_1)\le\ell-1$,
it follows that the depth of $T$ is at least $\log_2 \frac{d}{2} + 1 = \log_2 d$.
\end{proof} 

We now recall the notion of graph tree-depth. 
We use $\cl(T)$ to denote the transitive closure of a rooted tree $T$, 
i.e., the graph with vertex set $V(T)$ and an edge connecting each pair of vertices $u$ and $v$ such that $u$ is an ancestor of $v$ in $T$. 
\begin{definition} 
\label{d.treedepth} 
The \term{tree-depth} $\td(G)$ of a graph $G$ is the smallest possible depth of a rooted tree $T$ with the same vertex set as $G$ 
such that $G \se \cl(T)$. 
Such a tree $T$ is called an \term{optimal} tree-depth decomposition of $G$. 
\end{definition}
We next relate the branch-depth of a graphic matroid to the tree-depth of the underlying graph. 
\begin{proposition}
\label{p.relationtoG} 
The branch-depth of a graphic matroid $M(G)$ is at most $\td(G)$. 
\end{proposition}
\begin{proof}
Let $G$ be a graph on $n$ vertices and let $M := M(G)$ be the corresponding graphic matroid. 
We proceed by induction on $n$. If $n=1$ or $n=2$, the claim holds.
If $G$ is not 2-connected, let $G_1,\ldots,G_k$ be its 2-connected components (blocks).
Since the tree-depth is a minor monotone parameter, each $G_i$ has tree-depth at most $\td(G)$ and
the matroid $M(G_i)$ has a depth-decomposition with depth at most $\td(G)$ by induction.
Since the matroid $M$ is the disjoint union of the matroids $M(G_1),\ldots,M(G_k)$,
a depth-decomposition of $M$ can be obtained by identifying the roots of depth-decompositions of $M(G_1),\ldots,M(G_k)$.
The depth of such a depth-decomposition is at most $\td(G)$ and the claim follows.
So, we assume that $G$ is 2-connected in the rest of the proof.

Let $T$ be an optimal tree-depth decomposition of $G$.
We construct a \decomp\ $(T, f)$ of $M$ as follows. 
The function $f$ maps an element $e\in M$ to the end vertex of $e$ that is farther from the root of $T$.
We verify the two conditions from Definition \ref{d.decomp}. 
Since $G$ is connected, we indeed have $r(M) = n - 1 = |V(T)| - 1$ as required by the condition (1). 
Consider a subset $X \se M$.
We show that $r(X) \le \|T^*(X)\|$ to establish the condition (2).
We may assume that $X$ has no circuit:
if $X$ contained a circuit,
removing an element from a circuit of $X$ would not change $r(X)$ and it could not increase $\|T^*(X)\|$.
So, we assume that $X$ is independent and $r(X)=|X|$.

Let $X_1,\ldots,X_k$ be the edge sets of the connected components of the graph $(V(G),X)$, and
let $U_1,\ldots,U_k$ be the their vertex sets.
Further, let $U'_i$ be the set $f^{-1}(X_i)$.
Note that the sets $U_1,\ldots,U_k$ are disjoint and $U'_i$ is a subset of $U_i$.
Since $(U_i,X_i)$ is connected, the set $U_i$ contains a unique vertex closest to the root, and
$U'_i$ is equal to $U_i$ with the vertex closest to the root removed.
Since the subtree $T^*(X)$ contains an edge from each vertex of $U'_i$ to its parent, and
the sets $U_1,\ldots,U_k$ are disjoint (and so are the sets $U'_1,\ldots,U'_k$),
$\|T^*(X)\|$ is at least $|U'_1|+\cdots+|U'_k|\ge |U_1|+\cdots+|U_k|-k$.
Since the rank of $X$ is equal to the sum of ranks of the sets $X_i$ and
$r(X_i)=|X_i|=|U_i|-1$, it follows that $r(X)=|U_1|+\cdots+|U_k|-k$.
This finishes the proof of the lemma.
\end{proof}

Note that the converse of Proposition \ref{p.relationtoG} does not hold.
The graphic matroids of an $(n+1)$-vertex star $K_{1,n}$ and an $(n+1)$-vertex path $P_{n+1}$ are isomorphic and both have branch-depth one
despite of the tree-depth of $K_{1,n}$ being one and the tree-depth of $P_{n+1}$ being $\lceil\log_2 n+1\rceil-1$.
Nevertheless, the following inequality holds for 2-connected graphs. 

\begin{proposition}
\label{p.relationtoG2} 
Let $G$ be a 2-connected graph with tree-depth $d$. 
Then, the branch-depth of a graphic matroid $M(G)$ is at least $\frac{1}{2}\log_2{d}$. 
\end{proposition}

\begin{proof}
Since $\td(G) = d$, the graph $G$ contains a cycle of length at least $\sqrt{d}$ by~\cite[Proposition 6.2]{bib-nom2012}. 
Therefore, by 
Proposition \ref{p.circuit}, $\bd(M(G)) \ge \frac{1}{2}\log_2 d$. 
\end{proof}

\subsection{Technical lemmas}

\label{s.technical}

In this section, we establish further properties of the branch-depth,
which are important to prove the correctness of the algorithm presented later. 
The following two claims follow directly from the definition of contracting an element of a matroid. 

\begin{lemma}
\label{l.tech.contr} 
Let $C$ be a circuit in a matroid $M$. 
Let $e \in C$.
If $|C|> 1$, then the set $C \setminus\{e\}$ is a circuit in $M / e$. 
\end{lemma}

\begin{proof}
By the definition of contracting an element, it follows that
$r_{M/e}(C\setminus e)=r_M(C)-1=|C|-2$.
On the other hand, if $X$ is a proper subset of $C$,
then $r_{M/e}(X)=r_M(X\cup\{e\})-1=|X|$.
Hence, $C\setminus\{e\}$ is a circuit in $M/e$.
\end{proof}

\begin{lemma}
\label{l.tech.decontr} 
Let $M$ be a matroid and $e$ an element of $M$.
If $C$ is a circuit of $M/e$, then $M$ has a circuit $C'$ such that $C'\supseteq C$.
\end{lemma}

\begin{proof}
First observe that any proper subset of $C$ is independent in $M$:
indeed, if $X$ is a proper subset of $C$,
then $r_M(X)\ge r_M(X\cup\{e\})-1=r_{M/e}(X)=|X|$.
If $C$ is not independent in $M$, then $C$ is a circuit.

Suppose that $C$ is independent in $M$.
We claim that $C\cup\{e\}$ is a circuit.
First, $r_M(C\cup\{e\})=r_{M/e}(C)+1=|C|$, i.e., $C\cup\{e\}$ is not independent.
Let $X$ be a subset of $C\cup\{e\}$.
We have already observed that $X$ is independent if $e\not\in X$.
If $e\in X$, then $r_M(X\cup\{e\})=r_{M/e}(X)+1=|X|+1$,
i.e., $X$ is independent. We conclude that $C\cup\{e\}$ is a circuit.
\end{proof}

When encountering a circuit, the algorithm is going to proceed by contracting one of its elements. 
The following lemma will be crucial for the analysis. 

\begin{lemma}
\label{l.tech.1}
Let $M$ be a connected matroid, $e$ an element of $M$ such that $M/e$ is disconnected, and let $M_1, \dots, M_k$ be the components of $M/e$. 
For every circuit $C$ of $M$ containing $e$, there exists $i \in \{1,\dots,k\}$ such that $C\subseteq M_i\cup \{e\}$.
\end{lemma}

\begin{proof}
Suppose that $C$ contains an element from $M_1$ and an element from $M_i$, $i>1$.
Let $M'_2$ be the union of $M_2,\ldots,M_k$, and let $D_1=C\cap M_1$ and $D_2=C\cap M'_2$.
Hence, we get the following
$$|D_1|+|D_2| =|C|-1 = r_M(C) = r_{M/e}(C \setminus e)+1\;.$$
Since $M_1,\ldots,M_k$ are the components of $M$,
it follows that $r_{M/e}(C\setminus \{e\})=r_{M_1}(D_1)+r_{M'_2}(D_2)$.
However, $C\setminus \{e\}$ is a circuit in $M/e$ by Lemma~\ref{l.tech.contr}.
Since $D_1$ and $D_2$ are proper subsets of $C\setminus \{e\}$,
both $D_1$ and $D_2$ are independent in $M/e$.
Hence, $D_1$ and $D_2$ are independent in $M_1$ and $M'_2$, respectively.
It follows that $r_{M/e}(C\setminus \{e\})=|D_1|+|D_2|$, which is impossible.
\end{proof} 

Lemma \ref{l.tech.contr} and Lemma \ref{l.tech.1} yield the following. 

\begin{lemma}
\label{l.circuit3}
Let $M$ be a connected matroid. 
Let $e$ be an element of $M$ such that $M/e$ is not connected and let $M_1,\dots,M_k$ be the components of $M/e$. 
For each $i=1,\dots,k$ there is a circuit $C_i$ in $M$ containing $e$ such that $C_i\subseteq M_i\cup\{e\}$.
\end{lemma}

\begin{proof}
Fix $i=1,\ldots,k$. Since $M$ is connected, there is a circuit containing any two elements of $M$,
in particular, $M$ has a circuit $C_i$ containing the element $e$ and an element  of $M_i$.
By Lemma~\ref{l.tech.1}, the circuit $C_i$ must be a subset of $M_i\cup\{e\}$.
\end{proof}

The following lemma allows us to find an obstruction to a small branch-depth. 
We utilize this lemma to show that Algorithm \ref{a.dec} always returns a \decomp\ of depth at most $4^{\bd(M)}$. 
Figure~\ref{fig-l.circuit} contains an illustration of the notation used in the lemma.

\begin{lemma}
\label{l.circuit}
Let $M$ be a matroid. Let $e_1,\ldots,e_k$ be distinct elements of $M$ and $C_0,C_1,\dots,C_k$ subsets of $M$ such that
$$
\begin{array}{rl}
|C_i|\ge 3 &\textrm{for $i=0,\dots,k$,}\\
C_{i-1}\cap C_{i}=\{e_i\} &\textrm{for $i=1,\dots,k$,}\\
C_i \cap C_j = \emptyset &\textrm{for $|i-j|\ge 2$.} 
\end{array}
$$
Let $e_0\in C_0\setminus\{e_1\}$ and $e_{i}'\in C_{i-1}\setminus\{e_{i-1},e_{i}\}$, $i=1,\dots,k$. 
Further, set 
$$
M_i := \left\{\begin{array}{ll}
M & \textrm{for $i=0$}, \\
M_{i-1}/(C_{i-1} \setminus \{e_i,e'_i\}) & \textrm{for $i=1,\dots,k$.}
\end{array}\right.
$$
If $C_i$ is a circuit in $M_i$ for $i=0,1,\dots,k$, then $M$ contains a circuit of length at least $k+3$ containing $e_0$.
\end{lemma}

\begin{figure}
\begin{center}
\epsfbox{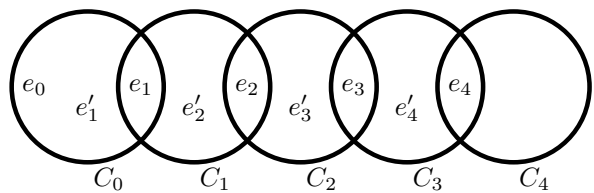}
\end{center}
\caption{The notation used in the proof of Lemma~\ref{l.circuit}.}
\label{fig-l.circuit}
\end{figure}

\begin{proof}
We prove the statement by induction on $k$. 
For $k=0$ it suffices to take the circuit $C_0$ itself. 

Let $k\ge 1$. 
By induction, $M_1 = M_0 / (C_0 \setminus \{e_1,e'_1\})$ has a circuit $C$ of length at least $k+2$ that contains $e_1$. 
Let $D=C\setminus \{e_1\}$. 
Since $C$ is a circuit, $D$ is independent in $M_1$ and thus in $M_0$. 
Also note $|D|\ge k+1$.

Let $N=M_0 / (C_0\setminus\{e_0,e_1,e'_1\})$. 
Since $C_0$ is a circuit in $M_0$, $\{e_0,e_1,e'_1\}$ is a circuit in $N$ by Lemma~\ref{l.tech.contr}. 
Furthermore, it holds that $M_1 = N/e_0$.
If $Y$ is a circuit in $N$, then there is a circuit $Y' \supseteq Y$ in $M$ by Lemma~\ref{l.tech.decontr}.
Therefore, it suffices to find a circuit of length at least $k+3$ in $N$. 

We will show that $D\cup\{e_0,e'_1\}$ or $D\cup\{e_0,e_1\}$ is a circuit in $N$. 
Since $D$ is independent in $N/e_0$, we get that $D\cup\{e_0\}$ is independent in $M$.  
We next show that
\begin{equation}
r_N(X\cup\{e_i,e_j\})=r_N(X\cup\{e_0,e_1,e'_1\})\label{eq-l.circuit}
\end{equation}
for any $e_i,e_j\in\{e_0,e_1,e'_1\}$, $e_i\ne e_j$ and for any set $X\subseteq N$.
This follows from the following application of the submodularity of the rank function:
$$
r_N(\{e_i,e_j\})+r_N(X\cup\{e_0,e_1,e'_1\})\le r_N(X\cup\{e_i,e_j\})+r_N(\{e_0,e_1,e'_1\}).
$$
Hence, for any proper subset $D' \subsetneq D$, we have
$$r_N(D'\cup\{e_0,e'_1\})=r_N(D'\cup\{e_0,e_1\})=r_{M_1}(D'\cup\{e_1\})+1 = |D'|+2,$$
where the last equality follows from the fact that $C=D \cup \{e_1\}$ is a circuit in $M_1$. 
Thus, both $D'\cup\{e_0,e'_1\}$ and $D'\cup\{e_0,e_1\}$ are independent in $N$. 
On the other hand, it also holds that
$$r_N(D\cup\{e_0,e'_1\})=r_N(D\cup\{e_0,e_1\})=r_{M_1}(D\cup\{e_1\})+1=|D|+1,$$
where the first equality follows from (\ref{eq-l.circuit}),
the second from $M_1=N/e_0$, and
the last from the fact that $C=D\cup\{e_1\}$ is a circuit in $M_1$.
Consequently, neither $D\cup\{e_0,e_1\}$ and $D\cup\{e_0,e'_1\}$ is independent in $N$.
To finish the proof,
it suffices to show that $D\cup\{e_1\}$ or $D\cup\{e'_1\}$ is independent in $N$. 
This can be shown using the submodularity of the rank function and (\ref{eq-l.circuit}) as follows:
$$
r_N(D\cup\{e_1\})+r_N(D\cup\{e'_1\})\ge r_N(D)+r_N(D\cup\{e_1,e'_1\})$$
$$=r_N(D)+r_N(D\cup\{e_0,e_1\})=2|D|+1.$$ 
The proof is now complete. 
\end{proof} 

We get the following corollary. 
\begin{corollary}
\label{l.circuit2}
Let $M$ be a matroid. 
If $C_0,C_1,\dots, C_k$ and $M_0, \ldots, M_k$ are as in Lemma \ref{l.circuit}, then the matroid $M$ contains a circuit of length at least $\sqrt{\sum_{i=0}^k |C_i|}$.
\end{corollary}

\begin{proof}
Let $t:=\sum_{i=0}^k |C_i|$. 
If $t \le (k+1)^2$, then by Lemma \ref{l.circuit} there is a circuit of length at least $k+3> \sqrt{t}$. 
On the other hand, if $t > (k+1)^2$, then there exists $i\in\{0,1,\dots,k\}$ such that $|C_i|\ge \frac{t}{k+1} > \sqrt{t}$.
\end{proof} 

\subsection{Approximating branch-depth}
\label{s.approximatingbd} 

We now present our polynomial-time algorithm for constructing a depth-de\-compo\-si\-tion of an oracle-given matroid $M$ with depth at most $4^{\mathrm{bd}(M)}$.
The pseudocode is given as Algorithm \ref{a.dec} in the form of a routine taking three parameters: a connected matroid $M$, one of its circuits $C$, and a non-loop element $e \in C$. 
For disconnected matroids we process the components individually and glue the resulting depth-de\-compo\-si\-tions by identifying their roots. 
Note that every connected matroid has a circuit and a non-loop element unless $|M|=1$. 

If the rank of $M$ is at most one,
which can be easily determined by checking the existence of a two-element independent set,
the routine returns the trivial \decomp,
which is either one-vertex or two-vertex rooted tree
with all the matroid elements mapped to the root or the non-root leaf, respectively,
depending on the rank of $M$.
Assume $r(M) \ge 2$.
If $|C| \le 2$, we find another circuit containing $e$ of size at least three.
The existence of such circuit in a connected matroid with rank at least two is implied by the definition of connectivity, and
it can be found in a polynomial time for example as follows.
Find using the greedy algorithm a base $B$ of $M$ not containing any element parallel to $e$ and remove from $B$ all the elements $f$ such that
$B\cup\{e\}\setminus\{f\}$ is not independent. The resulting set together with $e$ forms a circuit of length at least three.

If $|C| \ge 3$, we proceed by contracting $e$ in $M$ and analyzing the resulting matroid. 
If the resulting matroid is connected,
Algorithm~\ref{a.dec} calls itself recursively (Step 3 of Algorithm \ref{a.dec}) on the contracted matroid. 
Note that the connectivity of a matroid can be tested in polynomial time using the matroid intersection algorithm;
this algorithm can also be used to find the connected components if the matroid is not connected.
In case that the matroid is connected,
the recursive call is made for the contracted matroid, the circuit $C \setminus \{e\}$, and
an arbitrary element $e_1$ of $C\setminus\{e\}$.
Note that $e_1$ is not a loop in the contracted matroid since $|C|\ge 3$.
After the call is finished, we alter the resulting decomposition by adding a new vertex that becomes the root. 

If $M/e$ is not connected, the recursive calls are performed on each component separately (Step 4 and Step 5).
For the unique component containing $C\setminus\{e\}$ (see Lemma~\ref{l.tech.1}),
the call is performed with the component, the circuit $C \setminus \{e\}$, and
an arbitrary element $e_1$ of $C\setminus\{e\}$.
For other components,
the call is performed for $C'\setminus\{e\}$ where $C'$ is an arbitrary circuit of the original matroid
that contains $e$ and an element of the component;
the element of $C'\setminus\{e\}$ to perform the call with is chosen arbitrary.
The resulting decomposition is obtained by identifying the roots of the individual decompositions and
adding a new vertex that becomes the root of the whole decomposition.

It is easily verified that Algorithm~\ref{a.dec} finishes in time polynomial in the number of elements of the input matroid: 
if the recursive call in Step 2 is executed,
the next execution avoids this step and
performs one of the other recursive calls, which in turn lead to a decrease in the input size. 
If Step 3 is reached, only a single recursive call is made by the routine and the number of matroid elements is decreased by one. 
If Steps 4 and 5 are reached,
then the number of recursive calls equals the number of connected components and
the sum of the numbers of elements of the matroids passed to the calls
is one less than the number of elements of the original matroid.
It is easy to establish that the number of recursive calls
is at most quadratic in the number $n$ of elements of the original matroid, and
a more refined analysis can yield that the number of recursive calls is at most $O(n\log n)$.

We next establish that Algorithm~\ref{a.dec} produces a depth-decomposition of the input matroid. 

\begin{algorithm}
  \label{a.dec}
  \textbf{Algorithm \ref{a.dec}:} $\mbox{construct}(M,C,e)$ \\ 
  \KwIn{a connected matroid $M$, a circuit $C$ of $M$, and a non-loop element~$e\in C$}
  \KwOut{a \decomp\ of $M$}   
  \SetArgSty{textbb}
  \eIf{$r(M)=0$}{
	\nlset{Step 0} 
	\Return one-vertex tree with $f$ mapping all elements to the root\;
  } 
  {
    \eIf{$r(M)=1$}{
	\nlset{Step 1}
	\Return one-edge tree with $f$ mapping all elements to the leaf\;
   }
    {
      \eIf{$|C|=2$}{
	\nlset{Step 2} 
	choose a circuit $C'$ satisfying $|C'| \ge 3$ and $e \in C'$\; \Return $\mbox{construct}(M,C',e)$\;}
      {
        \eIf{$M/e$ is connected}
        {
          \nlset{Step 3}
          choose an element $e_1\in C\setminus \{e\}$\;
	  $(T',f'):=\mbox{construct}(M/e,C\setminus\{e\},e_1)$\;
          $T:=(V(T')\cup \{v\},E(T')\cup \{vv'\})$ where $v'$ is the root of $T'$\;
	  root $T$ at $v$\;
          $f(e) := v'; f(e') := f'(e')$ for $e' \ne e$\;
          \Return $(T,f)$\;
        }
        {
          \nlset{Step 4}
          \For{the component $M_0$ of $M/e$ containing $C\setminus e$}
          {
            choose an element $e_0\in C\setminus \{e\}$\;
	    $(T_0,f_0):=\mbox{construct}(M_0,C\setminus\{e\},e_0)$\;
          }
          \nlset{Step 5}
          \For{each component $M_i$ of $M/e$ disjoint from $C$}
          {
            choose a circuit $C_i$ of $M$ contained in $M_i\cup \{e\}$ that contains $e$\; 
            choose $e_i \in C_i\setminus \{e\}$\;
	    $(T_i,f_i) := \mbox{construct}(M_i,C_i\setminus\{e\},e_i)$\;
          }
          identify all the roots $v_i$ of $T_i$ into a single root $v'$, obtaining $T'$\;
          $T:=(v(T')\cup \{v\},E(T')\cup \{vv'\})$, root $T$ at $v$\;
          $f(e) := v', f(e_i) := f_i(e_i)$ for $e_i \in M_i$\;
          \Return $(T,f)$\;
        }
      }
    }
  }  
\end{algorithm}

\begin{lemma} 
Algorithm \ref{a.dec} returns a valid \decomp\ of $M$. 
\end{lemma} 

\begin{proof}
Let $M$ be the input matroid and $(T, f)$ the output of the Algorithm \ref{a.dec}.
Clearly, $T$ is a tree and $f$ a mapping from $M$ to $V(T)$.
Thus, we need to verify the two conditions from Definition \ref{d.decomp}.
We start with the condition (1), and verify it by induction on the number of recursive calls.
If Step 0 or 1 is reached, the tree $T$ clearly satisfies $r(M)=\|T\|$.
In Step 2, the algorithm is recursively evoked to a matroid of the same rank and
the returned tree $T$ has the number of edges equal to $r(M)$ by induction.
In Step 3, the routine is recursively called to a matroid with rank one smaller and
the returned tree is extended by a single edge; consequently, it also holds $r(M)=\|T\|$.
Finally, in Steps 4 and 5, the sum of the ranks of the matroids that the routine is called to is one smaller than $r(M)$ and
thus the output tree $T$ has $r(M)$ edges in this case, too.

We next establish the condition (2) from Definition \ref{d.decomp}, and
we again proceed by induction on the number of recursive calls.
Let $X$ be a non-empty subset of $M$; our aim is to show that $r(X)\le \|T^*(X)\|$.
If the depth-decomposition is constructed in Step 0 or 1, then the condition (2) clearly holds.
If the algorithm reached Step 2, the returned depth-decomposition is not modified and
the condition (2) holds by induction.
Suppose that the depth-decomposition was constructed in Step 3, i.e.,~the matroid $M / e$ is connected. 
From the induction, we get $r_{M/e}(X \setminus \{e\}) \le \|T^*(X)\|-1$,
since $T^*(X)$ includes one additional edge compared to the corresponding subtree of $T'$.
It follows that
$$r_M(X) \le r_M(X\cup\{e\})=r_{M/e}(X \setminus \{e\})+1\le \left\|T^*(X)\right\|.$$

It remains to analyze the case that
the depth-decomposition is constructed in Steps 4 and 5,
i.e., the case that the matroid $M / e$ is has components $M_0, \ldots, M_k$. 
By induction, we get $r_{M_i}(X \cap M_i) \le \left\|{T_i}^*(X \cap E(M_i))\right\|$, where $T_i$ is the depth-decomposition of $M_i$ returned by the recursive call for $M_i$. 
Since the resulting depth-decomposition is constructed by identifying the roots of $T_1, \ldots, T_k$ and connecting them to the new root, we get
$$r_M(X) \le r_M(X\cup\{e\})= 1 + \sum_{i=0}^k r_{M_i}(X \cap M_i) \le 1 + \sum_{i=0}^k \left\|T^*(X \cap M_i)\right\| = \left\|T^*(X)\right\|.$$
\qedhere
\end{proof}

We next analyze the depth of the tree returned by Algorithm~\ref{a.dec}.

\begin{lemma}
\label{l.algdepth}
Algorithm \ref{a.dec} returns a \decomp\ of $M$ with depth at most $4^{\mathrm{bd}(M)}$.
\end{lemma}

\begin{proof} 
Let $d$ be the depth of the \decomp\ $T$ returned by the algorithm for a matroid $M$. 
Let $r=v_0,v_1,\dots,v_d$ be a path in $T$ of length $d$ from the root to one of the leaves. 
It is easy to see that each vertex of $T$ that is not a leaf is the root of some subtree of $T$ during the execution of the algorithm. 
For $i=0,1,\dots,d-1$,
let $(M_i,C_i,e_i)$ be the matroid together with a circuit and an element of it such that
the algorithm creates the subtree rooted at $v_i$ and containing $v_{i+1}$ in the call with the parameters $(M_i,C_i,e_i)$.
Note that $M_0=M$ and $C_0=C$. In addition, $r(M_{d-1})=1$ and $|C_{d-1}|=2$.

For every $i=0,\dots,d-2$, precisely one of the following five cases occurs depending
whether the matroid $M_{i+1}$ was passed to the recursive call in Step 3, 4 or 5
during the execution of the call with the parameters $(M_i,C_i,e_i)$ and
whether the recursive call reached Step 2.
\begin{itemize}
\item $M_{i+1}=M_i/e_i$, $C_{i+1}=C_{i}\setminus \{e_i\}$ (Step 3),
\item $M_{i+1}=M_i/e_i$, $|C_i|=3$, $|C_{i}\cap C_{i+1}|=1$ (Step 3 followed by Step 2),
\item $M_{i+1}$ is a component of $M_i/e_i$, $C_{i+1}=C_i\setminus \{e_i\}$ (Step 4),
\item $M_{i+1}$ is a component of $M_i/e_i$, $|C_i|=3$, $|C_{i}\cap C_{i+1}|=1$ (Step 4 followed by Step 2),
\item $M_{i+1}$ is a component of $M_i/e_i$, $C_{i+1}\cap C_{i} = \emptyset$, but $C_{i+1}\cup \{e_i\}$ is a circuit in $M_i$ (Step 5).
\end{itemize} 
The sequence $C_0,\ldots,C_{d-1}$ consists of several runs where the next set is a subset of the preceding one;
the runs correspond to Steps 3 or 4.
Each run except for the last one is finished either by Step 3 or 4 followed by Step 2, or by Step 5.
Let $j_0,\ldots,j_k$ be the indices of the elements where the run starts,
i.e., the $i$-th run contains the elements with indices between $j_{i-1}$ and $j_i-1$ (inclusively).
Note that $j_0=0$, and set $\hat{C}_0 := C_{j_0} = C_0$ and $\hat{e}_0 := e_0$.
If the $i$-th run is finished by Step 5, 
let $\hat{C}_i := C_{j_i} \cup \{e_{j_i-1}\}$, $\hat{e}_i := e_{j_i-1}$ and $\hat{e}'_i$ any element of $C_{j_i-1}$ different from $e_{j_i-1}$.
If the $i$-th run is finished by Step 3 or 4 followed by Step 2,
let $\hat{C}_i := C_{j_i}$, let $\hat{e}_i$ be the unique element in $C_{j_i}\cap C_{j_i+1}$ and
let $\hat{e}'_i$ be the element of $C_{j_i-1}$ different from $e_{j_i-1}$ and from $\hat{e}_i$.

The circuits $\hat{C}_0,\ldots, \hat{C}_k$ together with the elements $\hat{e}_0,\ldots,\hat{e}_k$ and
the elements $\hat{e}'_1,\ldots,\hat{e}'_k$ fulfil the conditions of Lemma~\ref{l.circuit2}.
Since each $C_{j_i}$ can be contracted at most $|C_{j_i}|-2$ times before Step 2 or 5 is reached,
we have $\sum_{i=0}^k |\hat{C}_i| \ge d$.
By Corollary~\ref{l.circuit2}, $M$ has a circuit of length at least $\sqrt{\sum_{j=0}^k |\hat{C}_i|}\ge\sqrt{d}$.  
We conclude that $\bd(M) \ge \frac12\log_2 d$ by Proposition \ref{p.circuit}.
\end{proof} 

As a corollary of the above analysis, we get the following upper bound on the branch-depth of a matroid $M$.

\begin{corollary} 
The branch-depth of a finite matroid $M$ is at most $\ell^2$, where $\ell$ is the size of the largest circuit of $M$. 
\end{corollary} 

\begin{proof}
Apply Algorithm~\ref{a.dec} to $M$ and keep the notation of the proof of Lemma~\ref{l.algdepth}.
Since $M$ has a circuit of length at least $\sqrt{\sum_{j=0}^k |C_{i_j}|}\ge\sqrt{d}$,
it follows that $d\le\ell^2$. Hence, the branch-depth of $M$ is at most $\ell^2$.
\end{proof}

We obtain another corollary from the presented analysis of Algorithm~\ref{a.dec}.

\begin{corollary} 
\label{c.nicedec} 
There is a depth-decomposition $(T,f)$ and a base $B$ of $M$ such that $f|_B$ is a bijection 
between $B$ and the non-root vertices of $T$, and $f(e)$ is a leaf for every $e\not\in B$. 
Furthermore, the depth of $T$ is at most $4^{\bd(M)}$.
\end{corollary} 

\begin{proof}
We prove that the pair $(T,f)$ returned by Algorithm~\ref{a.dec} satisfies the statement.
The depth of $T$ is at most $4^{\bd(M)}$ by Lemma~\ref{l.algdepth}, and
we prove the existence of $B$ by induction on the number of recursive calls.
Let $C$ and $e$ be the parameters used to execute Algorithm~\ref{a.dec}.
If the algorithm constructed $(T,f)$ in Step 0 or 1, then the statement is obvious.
If $(T,f)$ is constructed in Step 2, the existence of $B$ follows from induction.
If $(T,f)$ is constructed in Step 3, the base from the recursive call together with $e$ forms a base $B$ of $M$.
Finally, if  $(T,f)$ is constructed in Steps 4 and 5,
the bases from the recursive calls together with $e$ form a base $B$ of $M$.
Observe that in the last two cases
the elements of the base $B$ are mapped by $f$ to different non-root vertices, and
all the other elements are mapped by $f$ to leaves.
\end{proof}

\section{Limits of representable matroids} 
\label{sect-modeling}

This section is devoted to the proof of Theorem~\ref{thm-branchdepth}. 
We first describe an encoding of first-order properties of a matroid in a rooted forest and we then employ Theorem \ref{limForests} 
to find a limit modeling.

Let $q$ be a fixed prime power, and let $M$ be a finite matroid which is representable over a finite field $\F_q$. 
We identify the elements of $M$ with the corresponding vectors. 
Let $d=\bd(M)$ be the branch-depth of $M$. 
Recall that by Corollary \ref{c.nicedec} there exists a depth-decomposition 
$(T,f)$, where $T$ has depth at most $D:=4^d$,
such that 
there is a base $B$ of $M$ with $f|_B$ being a bijection between $B$ and 
the non-root elements of $T$, and where $f(e)$ is a leaf for every $e\not\in B$.

We now construct a vertex-colored forest $F=F(M)$ from $(T,f)$ 
such that 
each component of $F$ is a rooted tree of depth at most $D$,   
the colors are $D$-tuples of elements of the finite field $\F_q$,   
and there is a bijection between the vertices of $F$ and the elements of $M$.  
The construction proceeds as follows. 
For every non-root vertex $v$ of $T$, let $b_v$ be the unique element of the base $B$ that is mapped by $f$ to $v$. 
Let $v\in L(T)$ be a leaf of $T$, say of depth $t$, and let $P(v)=(v_0,v_1,v_2,\ldots,v_t)$ 
be the unique path from the root $v_0$ to $v_t=v$. 
Every element  $e\not\in B$ that is mapped by $f$ to $v$ is in the closure of $b_{v_1},\ldots,b_{v_t}$
since $r_M(\{e,b_{v_1},\ldots,b_{v_t}\})\le \|T^*(\{v_1,\ldots,v_t\}\|=t$.
Consequently,
if matroid elements are viewed as vectors, $e$ is contained in the linear hull of $b_{v_1},\ldots,b_{v_t}$, and
it can be expressed as their linear combination, i.e., $e=\sum_{i=1}^t \alpha_i b_{v_i}$. 
For every such $e\in f^{-1}(v_t)\setminus B$, 
we attach a new leaf with its parent being $v_t$ and color it with the color 
$(\alpha_1,\ldots,\alpha_t,0,\ldots,0)$. 
Each vertex of the original tree $T$ is colored with the $d$-th unit vector $e_d$, where $d$ is its depth in $T$. 

We next delete the root and obtain a rooted forest with subtrees $T_1,\ldots, T_\ell$.  
Each of the trees $T_1,\ldots, T_\ell$ has height at most $D$. 
We call the resulting forest $F=F(M)$ with coloring $c_M:V(F)\to \F_q^D$ a {\em forest representation of $M$}. 
It follows from the construction that there is a bijection $g$ between the elements of $M$ and the vertices of $V(F)$. 

A forest representation of $M$ is not unique since there can be several different depth-decompositions of $M$. 
However, a forest representation $(F,c)$ uniquely determines the matroid. 
The elements of the matroid are the vertices of $F$.  
A set of $k$ vertices $\{v_1,\ldots,v_k\}$ is independent if the following holds.  
For $1\leq i \leq k$, let $d_i$ be the depth of $v_i$ in $F$, 
and let $P(v_i)=(w_1^{(i)},\ldots,w_{d_i+1}^{(i)})$ be the unique path from the root of the tree containing $v_i$ 
to $v_i=w_{d_i+1}^{(i)}$. We associate $v_i$ with a formal sum  
$\sum_{j=1}^{d_i} \alpha_j^{(i)} w_j^{(i)}$, 
where $(\alpha_1^{(i)},\ldots,\alpha_{d_i}^{(i)},0,\ldots,0) = c(v_i)$ is the color of $v_i$. 
The set $\{v_1,\ldots,v_k\}$ is linearly independent if and only if there exists no non-trivial $k$-tuple $(x_1,\ldots,x_k)\in \F_q^k$ 
such that 
$$ x_1\cdot \left(\sum_{j=1}^{d_1} \alpha_j^{(1)} w_j^{(1)}\right) + \cdots + x_k\cdot \left(\sum_{j=1}^{d_k} \alpha_j^{(k)} w_j^{(k)}\right) = 0.$$
Note that whether $\{v_1,\ldots,v_k\}$ is independent is determined by the subforest 
$F^\ast=F^\ast(v_1,\ldots,v_k)$, which is formed by the union of the paths $P(v_i)$, and by the colors $c(v_1),\ldots,c(v_k)$. 

Let $\lambda_D$ be the language describing rooted forests
with vertices colored with 
$\F_q^D$, i.e., $\lambda_D$ contains a binary relation representing the parent relation 
and $q^D$ unary symbols for the colors of the vertices. 
We now introduce an interpretation scheme $\bI$ of exponent one of
$\lambda_M$-structures (matroids) in $\lambda_D$-structures (vertex-colored rooted forests).
Let $\Theta(x)\;\equiv\; x=x$ and $\theta_E(x,y)\;\equiv\; x=y$.
For $k\geq 1$, we define $\theta_k\in \mathrm{FO}_k(\lambda_D)$ to encode the $k$-ary independence 
operator $I_k$ in the following way. 
Let $v_1,\ldots,v_k$ be vertices in $F$. 
For every $1\leq i\leq k$, let $P(v_i)$ and $c(v_i)$ be defined as above. 
Since each path $P(v_i)$ has length at most $D$, 
and since the number of colors is at most $q^D$, 
there are finitely many possibilities how $F^*(v_1,\ldots,v_k)$ can look.
The subforest $F^*(v_1,\ldots,v_k)$ completely determines
whether there exists a non-trivial linear combination of the formal sums  $\sum_{j=1}^{t_i} \alpha_j^{(i)} w_j^{(i)}$ that is zero.
We set $\theta_k(v_1,\ldots,v_k)$ to be the $\lambda_D$-formula that says that
$F^*(v_1,\ldots,v_k)$ is such that there exists no non-trivial zero linear combination;
the above reasoning yields that $\theta_k(v_1,\ldots,v_k)$ should say that $F^*(v_1,\ldots,v_k)$ is isomorphic to one of finitely many vertex-colored rooted forests,
i.e., there exists a $\lambda_D$-formula with these properties.

The next lemma follows from the construction and the definition of an interpretation. 
\begin{lemma}\label{aux558}
Let $M$ be a finite $\F_q$-representable matroid of branch-depth at most $d$, 
 and let $(F,c_M)$ be a forest representation of $M$. 
Then $\bI((F,c_M))\cong M$. 
\end{lemma}

We are now ready to prove our main theorem. 
\begin{proof}[Proof of Theorem \ref{thm-branchdepth}]
Let $q$ be a prime power, $d$ an integer, and  
$(M_n)_{n\in N}$ a first order convergent sequence of matroids such that each matroid $M_n$
is representable over the field $\F_q$ and has branch-depth bounded by $d$. 
Let $(F_n,c_n)_{n\in\N}$ be the corresponding sequence of forest representations of depth at most $D:=4^d$.  
Recall that each $c_n$ uses at most $q^{D}$ colors. 

By compactness (see~\cite{bib-folim1,bib-folim2}),
we can assume that $(F_n,c_n)_{n\in \N}$ is first order convergent (otherwise, we pick 
a first order convergent subsequence). 
By Theorem \ref{limForests}, $(F_n,c_n)_{n\in \N}$ has a limit modeling $(\bF,{\bf c})$. 
Note that $\bF$ has depth at most $D$ and ${\bf c}$ is an $\F_q^D$-coloring.

We now show that the assumptions of Lemma \ref{aux671} hold.  
Since $F_n\models \theta_0(v)$ for any $n\in \N$ and any vertex $v$ of $F_n$, 
the condition $(1)$ of Lemma \ref{aux671} trivially holds. 
Finally, the condition $(2)$ holds by Proposition \ref{aux089} and since there is a bijection 
$g_n$ between the elements of $M_n$ and the vertices of $F_n$, for any $n\in \N$. 

It follows from Lemma \ref{aux671} that 
$\M:= \bI((\bF,{\bf c}))$ is a limit modeling for the sequence 
$(\bI(F_n,c_n))_{n\in\N}$, and thus of the sequence 
$(M_n)_{n\in\N}$, by Lemma~\ref{aux558}.
It remains to prove that $\M$ is an infinite matroid. 
For $k\in \N$, let $\cC_k$ be the $k$-element subsets $\{x_1,\ldots,x_k\}\se \M$ 
of $\M$ that satisfy 
$$\M \models \neg I_k(x_1,\ldots ,x_k) \land \bigwedge_{1\leq i \leq k} I_{k-1}(x_1,\ldots,x_{i-1},x_{i+1},\ldots,x_k).$$ 
We show that $\cC:=\bigcup_{k\geq 1} \cC_k$ is a collection of circuits of $\M$. 

First, note that for every $n\in \N$, since the branch-depth of each $M_n$ is at most $d$, 
the length of every circuit of $M_n$ is at most $2^d$, by Proposition \ref{p.circuit}. 
It follows that the length of every circuit in $\M$ (i.e.,~the size of each element in $\cC$) is at most $2^d$ 
(since, e.g.,~the first order formula stating that there is no circuit of size $2^d+1$ 
holds in each $M_n$ and hence in $\M$). 
Thus, the axioms (C1), (C2) and (C3') trivially hold since they are equivalent to a finite set 
of first order axioms.
Finally, since all circuits of $\M$ are finite, $\M$ is finitary, and thus, as mentioned in Subsection \ref{subs:InfMatroids}, 
it is also a finitary matroid. 
\end{proof} 

\section{Non-existence of matroid limit modelings}
\label{sect-non-exist}

In this section, we show that neither of the assumptions in Theorem \ref{thm-branchdepth}  
can be removed.  
In particular, we prove the following two results. 
\begin{theorem} 
\label{t.non}
There exists a first order convergent sequence of $\Q$-representable matroids 
of rank three 
that has no limit modeling. 
\end{theorem} 
\begin{theorem} 
\label{t.non2}
There exists a first order convergent sequence of binary matroids that has no limit modeling.
\end{theorem}

\subsection{Matroids with rank three}\label{fuenfPunkt1}

We now construct a sequence of matroids of rank three, each representable over the field of rationals, that has no 
limit modeling. 
We start with describing a procedure to convert a graph $G$ to a $\Q$-representable matroid of rank at most three 
in a way that the first order properties of $G$ are preserved. 
The existence of a first-order convergent sequence of graphs without a limit modeling (e.g.,~a sequence of Erd\H{o}s-R\'enyi random graphs 
has this property with probability one~\cite[Lemma 18]{bib-folim2}) 
allows us to deduce that there exists a first order convergent sequence of such matroids that does not have a limit modeling.

Let $G$ be a graph on $n$ vertices. 
We construct a matroid $M(G,k)$, where $k\in\N$ is a parameter. 
We first give the construction for $k=1$. 
The construction is illustrated in Figure \ref{fig.rank3}.  
\begin{figure}
\begin{center}
\epsfbox{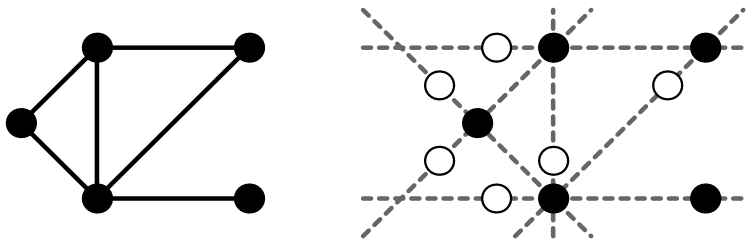}
\end{center}
\caption{Encoding a general graph $G$ in a rank 3 matroid $M(G, 1)$.}
\label{fig.rank3}
\end{figure}
Let $P=\{p_v: v\in V(G)\}$ be a set of $n$ points in the Euclidean plane $\R^2$ in general position. 
For every edge $uv\in E(G)$, add a new point $p_{uv}$ to the set $P$ on the line through the points $p_u$ and $p_v$ 
such that $p_{uv}$ does not lie on any line passing through another pair of the points (of $P$ and the newly added points).
The ground set of the matroid $M(G,1)$ is $P$. 
The bases $\mathcal{B}$ of $M(G,1)$ are all triples 
$\{x,y,z\}$ of $M(G,1)$ such that $x$, $y$, and $z$ do not lie on a line. 
Equivalently, the circuits are all sets of size four 
and all triples $\{p_u,p_v,p_{uv}\}$ where $uv\in E(G)$. 
The matroid $M(G,1)$ has rank at most 3 and it is $\Q$-representable. 
Note that the matroids obtained for a different initial choice of $P$ are isomorphic.

For $k\geq 2$, $M(G, k)$ is constructed from $M(G,1)$ by adding $k-1$ parallel elements to $p_v$ for every $v\in V(G)$. 
This class of $k$ parallel elements corresponding to $p_v$ is denoted by $C(v)$, and we say that an element in $C(v)$ {\em represents} 
the vertex $v$. We call an element of $M(G, k)$ a {\em vertex element} if it is contained in some $C(v)$, and an {\em edge element} otherwise.  

Let $k\geq 2$. To show that first order definable properties of $G$ are preserved in $M(G,k)$ let us define 
an interpretation scheme $\bI$ of $\lambda_G$-structures in $\lambda_M$-structures of exponent one. 
Note first that the formula $D_2\equiv\neg I_2 \in\mathrm{FO}_2(\lambda_M)$,
which captures the dependence of pairs of elements,  
defines an equivalence relation on the ground set of any such matroid $M(G,k)$. 
Further, note that edge elements can be distinguished from vertex elements in $M(G,k)$ since 
$x\in M(G,k)$ is a vertex element if and only if  
$M(G,k) \models \, \exists\, y\, \left( x\neq y \land \neg I_2(x,y)\right)$. 
We set 
$\theta_E(x,y) \equiv D_2(x,y) \equiv\neg I_2(x,y)$ and
$\theta_0 \equiv \big(\exists\, y\, \left( x\neq y \land \neg I_2(x,y)\right)\big)$. 
Finally, we define the formula $\theta_1\in FO_2(\lambda_M)$ to represent the edge-relation as  
$$\theta_1 (x,y)\  \equiv\  I_2(x,y)\land \exists z \left( I_2(x,z) \land I_2(y,z) \land \neg I_3(x,y,z)\right).$$ 
It is clear that $\theta_0$ and $\theta_1$ are compatible with $\theta_E$. 
That is, $\bI$ defined by the formulas $\theta_E$, $\theta_0$, and $\theta_1$ is an interpretation scheme of exponent one. 
The construction and the definition of the interpretation $\bI(M(G,k))$ yield the following. 

\begin{lemma}\label{aux712}
For any graph $G$ and $k\geq 2$, it holds that 
$\bI(M(G,k)) \cong G$. 
\end{lemma}

We are ready to prove Theorem \ref{t.non}. 

\begin{proof}[Proof of Theorem \ref{t.non}]
Let $(G_n)_{n\in\N}$ be a first order convergent sequence of graphs that does not have a limit modeling. 
Consider now the sequence $(M_n)_{n\in \N}$ where $M_n= M(G_n, |G_n|^2)$. 
We may assume that $(M_n)_{n\in\N}$ is first order convergent 
(otherwise, we consider a first order convergent subsequence, which exists by compactness). 
Let $\bI$ be the interpretation scheme defined earlier, and let 
$\tilde{\bI}:\mathrm{FO}(\lambda_G)\to \mathrm{FO}(\lambda_M)$ be the mapping from Proposition \ref{aux089}. 
We claim that the conditions in Lemma \ref{aux671} are satisfied for the sequence $(M_n)_{n\in\N}$. 
Let $x\in M_n$ be an element chosen uniformly at random. 
Then the probability that $x$ is an edge element is 
$\frac{|E(G_n)|}{|E(G_n)|+ |G_n|^3}\leq \frac{1}{|G_n|}$. 
Since $|G_n|$ tends to infinity, 
$\lim_{n\rightarrow \infty} \langle \theta_0,M_n\rangle = 1$ 
and the first condition in Lemma~\ref{aux671} holds. 

Now for $\ell\in \N$, consider $\phi\in\mbox{FO}_\ell(\lambda_G)$. 
Let $(e_1,\ldots,e_\ell)$ be an $\ell$-tuple chosen uniformly at random from $M_{n}$. 
Since  $\lim_{n\rightarrow \infty} \langle \theta_0,M_n\rangle = 1$,  
the probability that at least one $e_j$ is not a vertex element in $M_{n}$ 
tends to zero. 
Since in $M_n$ the equivalence classes $C(v)$ for $v\in V(G_n)$ have all the same size, 
it follows from Proposition~\ref{aux089} and Lemma~\ref{aux712} that 
$$\langle \phi,G_n\rangle = \langle \phi,\bI(M_n)\rangle = \langle \tilde{\bI}(\phi),M_n\rangle +o(1).$$
So, the second condition in Lemma~\ref{aux671} also holds. 

Assume now that $(M_{n})_{n\in \N}$ has a matroid limit modeling $\M$. 
By Lemma \ref{aux671}, $\G:= \bI(\M)$ is a limit modeling for the sequence 
$(\bI(M_n))_{n\in\N}$, that is for $(G_n)_{n\in\N}$ by Lemma~\ref{aux712}, 
a contradiction.
\end{proof}

\subsection{Binary matroids}

We describe an interpretation scheme of graphs in binary matroids. 
The argument is largely analogous to the one presented in Subsection~\ref{fuenfPunkt1}.

Let $G=(V,E)$ be a graph on $n$ vertices, and let $k\in \N$ be a parameter. 
We define a binary matroid $M'=M'(G,k)$ of rank $|V|$ in the following way. 
The matroid
$M'$ contains $k$ distinct elements represented by the unit vector $e_v\in \F_2^V$ for every vertex $v\in V$, and
$M'$ contains an element represented by $e_u+e_v$ for every edge $uv\in E$.
Recall the interpretation scheme $\bI$ of exponent one defined in Subsection~\ref{fuenfPunkt1}, and
observe that observe that the graph $\bI(M'(G,k))$ is isomorphic to $G$.
Theorem~\ref{t.non2} can now be proven in a way completely analogous to the proof of Theorem~\ref{t.non}.

\section{Concluding remarks}

The tree-depth of graphs is important in relation to testing graph properties in a fixed parameter way.
It is also important with respect to the structure of graphs in general. For example, for every $d$, there exists a finite set $\GG$ of graphs
with tree-depth at most $d$ such that each graph of tree-depth at most $d$ is homomorphically equivalent to one of the graphs in $\GG$.
Naturally, one may ask whether some of these results can be generalized to matroids using the branch-depth parameter introduced in this paper.

Independently of us, Matt DeVos and Sang-il Oum (private communication) were considering another tree-depth like parameter for matroids,
which was inspired by the work of Dittmann and Oporowski~\cite{bib-dittmann02+}.
They define a contraction-depth of a matroid recursively as follows.
A matroid that consists of loops and co-loops only has contraction-depth 0.
For other matroids, the contraction-depth of $M$ is the smallest $k$ that
there exists an element $e$ such that each component of $M/e$ has contraction-depth at most $k-1$.
As in the case of branch-depth, the contraction-depth of a matroid is both lower and upper bounded by the length of its longest circuit.
They have also been working on the variant of the parameter called deletion-depth defined similarly and
on generalizations for arbitrary connectivity functions,
which are of interest in relation to problems from discrete optimization.

\end{document}